\documentclass[12pt]{amsart}
\usepackage{a4wide,color}
\allowdisplaybreaks

\let\pa\partial
\let\na\nabla
\let\eps\varepsilon

\newcommand{\R}{{\mathbb R}}
\newcommand{\T}{{\mathbb T}}
\newcommand{\diver}{\operatorname{div}}

\newtheorem{theorem}{Theorem}
\newtheorem{lemma}[theorem]{Lemma}

\begin{document}

\title[An asymptotic limit of a Navier-Stokes system]{An asymptotic limit of a
Navier-Stokes system with capillary effects}

\author[A. J\"ungel]{Ansgar J\"ungel}
\address{Institute for Analysis and Scientific Computing, Vienna University of
	Technology, Wiedner Hauptstra\ss e 8--10, 1040 Wien, Austria}
\email{juengel@tuwien.ac.at}

\author[C.-K. Lin]{Chi-Kun Lin}
\address{Department of Applied Mathematics and Center of Mathematical Modeling
and Scientific Computing, National Chiao Tung University, Hsinchu 30010, Taiwan}
\email{cklin@math.nctu.edu.tw}

\author[K.-C. Wu]{Kung-Chien Wu}
\address{Department of Pure Mathematics and Mathematical Statistics,
University of Cambridge, Wilberforce Road, Cambridge CB3 0WB, United Kingdom}
\email{kcw28@dpmms.cam.ac.uk}

\date{\today}

\thanks{The first author acknowledges partial support from
the Austrian Science Fund (FWF), grants P22108, P24304, and I395 and
from the National Science Council of Taiwan (NSC) through the Tsungming Tu Award.
The second author was supported by the NSC,
grant NSC101-2115-M-009-008-MY2. The third author acknowledges partial
support from the Tsz-Tza Foundation of the Institute of Mathematics,
Academia Sinica (Taipei, Taiwan). He thanks Cl\'ement Mouhot for his
kind invitation to visit the University of Cambridge during the academic years 2011-2013. 
Part of this work was written during the stay of the first author
at the Center of Mathematical Modeling and Scientific Computing,
National Chiao Tung University (Hsinchu, Taiwan), and at the Institute of
Mathematics, Academia Sinica (Taipei, Taiwan); he thanks Chi-Kun Lin
and Tai-Ping Liu for their kind hospitality}

\begin{abstract}
A combined incompressible and vanishing capillarity limit in the barotropic
compressible Navier-Stokes
equations for smooth solutions is proved. The equations are considered
on the two-dimensional torus with well prepared initial data.
The momentum equation contains a rotational term originating from a Coriolis force,
a general Korteweg-type tensor modeling capillary effects, and a density-dependent
viscosity. The limiting model is the viscous quasi-geostrophic equation
for the ``rotated'' velocity potential. The proof of the singular limit
is based on the modulated energy method with a careful choice of the
correction terms.
\end{abstract}

% \paragraph{Keywords:}
\keywords{Incompressible limit, vanishing capillarity limit, modulated energy, Navier-Stokes-Korteweg model,
quasi-geostrophic equation.}

% \paragraph{AMS classification:}
\subjclass[2000]{35B25, 35Q35, 35Q40.}
% 65M15 error bounds

\maketitle

%%%%%%%%%%%%%%%%%%%%%%%%%%%%%%%%%%%%%%%%%%%%%%%%%%%%%%%%%%%%%%%%%%%%%%%%%%%%%%

\section{Introduction}

The aim of this paper is to prove a combined incompressible and vanishing capillarity limit for a two-dimensional
Navier-Stokes-Korteweg system, leading to the viscous quasi-geostrophic equation.
We consider the (dimensionless) mass and momentum equations for the particle
density $\rho(x,t)$ and the mean velocity $u(x,t)=(u_{1}(x,t),u_{2}(x,t))$ of a fluid
in the two-dimensional torus $\T^2$:
\begin{align}
  &\pa_t\rho + \diver(\rho u) = 0\quad\mbox{in }\T^2,\ t>0, \label{rho} \\
  & \pa_t(\rho u) + \diver(\rho u\otimes u) + \rho u^\perp + \na p(\rho)
  = \diver (K + S), \label{rho.u}
\end{align}
with initial conditions
$$
  \rho(\cdot,0)=\rho^0, \quad u(\cdot,0)=u^0\quad\mbox{in }\T^2.
$$
Here, $\rho u^\perp$ describes the Coriolis force, $u^\perp=(- u_{2},u_{1})$,
the function $p(\rho)=\rho^\gamma/\gamma$ with $\gamma>1$ denotes the pressure
of an ideal gas obeying Boyle's law,
$K$ is the Korteweg-type tension tensor and $S$ the viscous stress tensor. 

More precisely, the free surface tension tensor is given by
$$
  \diver K = \kappa_0\rho\na(\sigma'(\rho)\Delta\sigma(\rho)),
$$
where $\kappa_0>0$, which can be written in conservative form as
\begin{equation}\label{cons}
  \diver K = \kappa_0\diver\left(\left(\Delta S(\rho)
  - \frac12S''(\rho)|\na\rho|^2\right)
  {\mathbb I} - \na\sigma(\rho)\otimes\na\sigma(\rho)\right),
\end{equation}
where $S'(\rho)=\rho\sigma'(\rho)^2$,
$\sigma(\rho)$ is a (nonlinear) function, and
${\mathbb I}$ denotes the unit matrix in $\R^{2\times 2}$.
For a general introduction and the physical background
of Navier-Stokes-Korteweg systems, we refer to \cite{BDL03, DuSe85, Jue10}. 
In standard Korteweg models, $\kappa(\rho)=\sigma'(\rho)^2$ defines the
capillarity coefficient \cite[Formula (1.29)]{DuSe85}.
In the shallow-water equation, often $\sigma(\rho)=\rho$ is used such that
$\diver K=\rho\na\Delta\rho$ (see, e.g., \cite{BrDe03,Mar07}).
Bresch and Desjardins \cite{BrDe04} employed general functions $\sigma(\rho)$
and suitable viscosities allowing for additional energy estimates 
(also see \cite{Jue11}).
If $\sigma(\rho)=\sqrt{\rho}$, the third-order term can be interpreted as
a quantum correction, and system \eqref{rho}-\eqref{rho.u} (without the
rotational term) corresponds to the so-called quantum Navier-Stokes model,
derived in \cite{BrMe10} and analyzed in \cite{Jue10}.

The viscous stress tensor is defined by
$$
  \diver S = 2\diver(\mu(\rho)D(u)),
$$
where $D(u)=\frac12(\na u+\na u^\top)$ and $\mu(\rho)$ denotes the density-dependent
viscosity. Often, the viscosity in the Navier-Stokes model is assumed to be constant
for the mathematical analysis \cite{Fei04}.
Density-dependent viscosities of the form $\mu(\rho)=\rho$ were chosen in
\cite{BrDe03} and were derived, in the context of the quantum Navier-Stokes model,
in \cite{BrMe10}. The choice $\mu(\rho)=\sigma(\rho)$ allows one to exploit
a certain entropy structure of the system \cite{BrDe04}.

Without capillary effects, system \eqref{rho}-\eqref{rho.u} reduces to the
viscous shallow-water or viscous Saint-Venant equations, whose inviscid version
was introduced in \cite{Sai71}. The viscous model was formally derived from
the three-dimensional Navier-Stokes equations with a free moving boundary
condition \cite{GePe01}. This derivation was generalized later to varying river
topologies \cite{Mar07}.
The existence of global weak or strong solutions to the Korteweg-type
shallow-water equations was proved in \cite{BrDe04,BDL03,Has11,Has12,HsLi08}
under various assumptions on the nonlinear functions.
In \cite{BDL03}, the authors obtained several existence results of weak solutions under various assumptions concerning the density 
dependency of the coefficients. The notion of weak solution involves test functions depending on the density; 
this allows one to circumvent the vacuum problem. Duan et al.\ \cite{DLZ12} showed the existence of local classical solutions
to the shallow-water model without capillary effects.
For more details and references on the shallow-water system,
we refer to the review \cite{Bre09}.

The combined incompressible and vanishing capillarity limit studied in this work
is based on the scaling
$t\mapsto \eps t$, $u\mapsto \eps u$, $\mu(\rho)\mapsto \eps\mu(\rho)$
and on the choice $\kappa_0=\eps^{2\alpha}$ ($0<\alpha,\eps<1$), which gives
\begin{align}
  & \pa_t\rho_\eps + \diver(\rho_\eps u_\eps) = 0
  \quad\mbox{in }\T^2,\ t>0, \label{eq1} \\
  & \pa_t(\rho_\eps u_\eps) + \diver(\rho_\eps u_\eps\otimes u_\eps)
  + \frac{1}{\eps}\rho_\eps u_\eps^\perp
  + \frac{1}{\eps^2\gamma}\na(\rho_\eps^\gamma)
  - 2\eps^{2(\alpha-1)}\rho_\eps
  \na(\sigma'(\rho_\eps)\Delta\sigma(\rho_\eps)) \label{eq2} \\
  &\phantom{xx}{}= 2\diver(\mu(\rho_\eps) D(u_\eps)), \nonumber
\end{align}
with the initial conditions
\begin{equation}\label{ic}
  \rho_\eps(\cdot,0)=\rho_{\eps}^0, \quad u_{\eps}(\cdot,0)= u_{\eps}^0\quad\mbox{in }\T^2.
\end{equation}
The condition $\alpha<1$ is needed to control the capillary energy;
see the energy identity in Lemma \ref{lem.ee} below.
The local existence of smooth solutions to \eqref{eq1}-\eqref{eq2} is 
discussed in Appendix \ref{sec.loc}.

When letting $\eps\to 0$, it holds $\rho_\eps\to 1$ and $\rho_\eps u_\eps\to
\na^\perp\phi=(-\pa\phi/\pa x_2,\pa\phi/\pa x_1)$ in appropriate function spaces,
where $\phi$ solves the viscous quasi-geostrophic equation
\cite[Chapter 6]{Ped90} (see Section \ref{sec.aux} for details)
\begin{align}
  & \pa_t(\Delta\phi-\phi) + (\na^\perp\phi\cdot\na)(\Delta\phi) = \mu(1)\Delta^2\phi
  \quad\mbox{in }\T^2,\ t>0, \label{lim.eq} \\
  & \phi(\cdot,0) = \phi^0\quad\mbox{in }\T^2. \label{lim.ic}
\end{align}
The objective of this paper is to make this limit rigorous.
Our proof requires the (local) existence of a smooth solution to 
\eqref{lim.eq}-\eqref{lim.ic}, which is shown in Appendix \ref{sec.loc}. 
Several derivations of inviscid quasi-geostrophic equations have been published;
see, e.g., \cite{DeGr98,EmMa98,Sch87}. 
The reader is also referred to the monograph \cite{Maj03} for a 
more complete discussion of this model.  
The viscous equation was derived
rigorously for weak solutions from the shallow-water system in \cite{BrDe03}.
The proof is essentially based on the presence of the additional viscous part
$\diver(\rho\na u)$ and a friction term in the momentum equation.
The novelty of the present paper is that these expressions are not needed
and that more general expressions can be considered. In particular,
we allow for viscous terms of the type $\diver(\mu(\rho)D(u))$, and no friction
is prescribed.

In the following, we describe our main result.
In order to simplify the presentation, we assume that the nonlinearities
are given by power-law functions:
$$
  \sigma(\rho)=\rho^s, \quad \mu(\rho) = \rho^m \quad\mbox{for }\rho\ge 0,
$$
where $s>0$ and $m>0$. The exponents $s$ and $m$ cannot
be chosen freely; we need to suppose that
\begin{equation}\label{sm}
  0<s\le 1, \quad m = s+\frac12 \le \frac{\gamma+1}{2}.
\end{equation}
This assumption includes the quantum Navier-Stokes model $s=1/2$, $m=1$
and the shallow-water model with $s=1$, $m=3/2$.
Furthermore, we assume that the initial data are sufficiently regular
(ensuring the local-in-time existence of smooth solutions)
\begin{equation*}%\label{ic.reg}
  \rho_\eps^0\in H^k(\T^2),\ u_\eps^0\in H^{k-1}(\T^2), \
  \phi^0\in H^{k+1}(\T^2), \quad\mbox{where }k>2,
\end{equation*}
and that they are well prepared:
\begin{equation} \label{ic.prep}
  G_\eps(\phi_\eps^0)\to\phi^0, \
  \eps^{-1}(\rho_\eps^0-1) \to \phi^0, \
  \sqrt{\rho_\eps^0}u_\eps^0\to\na^\perp\phi^0,\
  \eps^{\alpha-1}\na\sqrt{\rho_\eps^0}\to 0
\end{equation}
in $L^2(\T^2)$ as $\eps \to 0$,
where $\rho_\eps^0=1+\eps\phi_\eps^0$ (this defines $\phi_\eps^0$),
\begin{equation}\label{G}
  G_\eps(\phi_\eps)
  = \frac{\sqrt{2}}{\eps}\mbox{sign}(\phi_\eps)\sqrt{h(1+\eps\phi_\eps)},
  \quad \rho_\eps = 1+\eps\phi_\eps,
\end{equation}
and the internal energy $h(\rho)$ is defined by
$h''(\rho)=p'(\rho)/\rho=\rho^{\gamma-2}$
and $h(1)=h'(1)=0$ (see \eqref{h} for an explicit
expression). Note that the convergence $\eps^{-1}(\rho_\eps^0-1) \to \phi^0$ in
$L^2(\T^2)$ implies that $G_\eps(\phi_\eps^0)\to\phi^0$ in $L^1(\T^2)$ if
$\rho_\eps^0$ is bounded in $L^\infty(\T^2)$ (see \eqref{est.G}).

\begin{theorem}\label{thm}
Let $0<\alpha<1$ and $\gamma>1$.
We suppose that \eqref{sm} holds and that the initial data satisfy
\eqref{ic.prep}. Furthermore, let $(\rho_\eps,u_\eps)$ be the
classical solution to \eqref{eq1}-\eqref{ic} and let $\phi$ be the classical
solution to \eqref{lim.eq}-\eqref{lim.ic}, both on the time interval $(0,T)$.
Then, as $\eps\to 0$,
\begin{align*}
  \rho_\eps\to 1 &\quad\mbox{in }L^\infty(0,T;L^{\gamma}(\T^2)), \\
  \rho_\eps u_\eps \to \na^\perp\phi &\quad\mbox{in }
  L^\infty(0,T;L^{2\gamma/(\gamma+1)}(\T^2)).
\end{align*}
Furthermore, if $s<\frac12$ and $\gamma\ge 2(1-s)$ or if $s=1$ and $\gamma\ge 2$,
\begin{align*}
  \rho_\eps\to 1 &\quad\mbox{in }L^\infty(0,T;L^{p}(\T^2)), \\
  \rho_\eps u_\eps \to \na^\perp\phi &\quad\mbox{in }
  L^\infty(0,T;L^{q}(\T^2)),
\end{align*}
for all $1\le p<\infty$ and $1\le q<2$.
\end{theorem}

The proof is based on the modulated energy method, first introduced
by Brenier in a kinetic context \cite{Bre00} and later extended to various
models, e.g.\ \cite{Bos07,BNP04,LiWu12}.
The idea of the method is to estimate, through its
time derivative, a suitable modification of the energy by introducing in the
energy the solution of the limit equation.
We suggest the following form of the modulated energy:
\begin{align}
  H_\eps(t) &= \int_{\T^2}\left(\frac{\rho_\eps}{2}|u_\eps-\na^\perp\phi|^2
  + \frac12|G_\eps(\phi_\eps)-\phi|^2
  + 2\eps^{2(\alpha-1)}|\na\sigma(\rho_\eps)|^2\right)dx \nonumber \\
  &\phantom{xx}{}+ 2\int_0^t\int_{\T^2}\mu(\rho_\eps)|D(u_\eps)-D(\na^\perp\phi)|^2 dx,
  \label{1.H}
\end{align}
These terms express the differences of the kinetic, internal, and Korteweg
energies as well as the viscosity. Differentiating the modulated energy
with respect to time and employing the evolution equations, elaborated
computations lead to the inequality
$$
  H_\eps(t) \le C\int_0^t H_\eps(s)ds + o(1), \quad t>0,
$$
where $o(1)$ denotes terms vanishing in the limit $\eps\to 0$, uniformly in time.
The Gronwall lemma then implies the result.

The paper is organized as follows. In Section \ref{sec.aux}, we derive the
energy identities for the shallow-water system and the quasi-geostrophic equation
and give a formal derivation of the latter model from the former one.
Theorem \ref{thm} is proved in Section \ref{sec.thm}. In the appendix,
we discuss the existence of local smooth solutions to \eqref{eq1}-\eqref{eq2}
and give an existence proof for local smooth solutions to \eqref{lim.eq}-\eqref{lim.ic}.

%%%%%%%%%%%%%%%%%%%%%%%%%%%%%%%%%%%%%%%%%%%%%%%%%%%%%%%%%%%%%%%

\section{Auxiliary results}\label{sec.aux}

In this section, we derive the energy estimates for \eqref{eq1}-\eqref{eq2}
and derive formally the quasi-geostrophic equation \eqref{lim.eq}.
Based on the definition $h''(\rho)=p'(\rho)/\rho$, $h(1)=h'(1)=0$, we can
give an explicit formula for this function:
\begin{equation}\label{h}
  h(\rho) = \frac{1}{\gamma(\gamma-1)}\big(\rho^\gamma-1-\gamma(\rho-1)\big),
  \quad \rho\ge 0.
\end{equation}
The energy identity for \eqref{eq1}-\eqref{eq2} is given as follows.

\begin{lemma}\label{lem.ee}
Let $(\rho_\eps,u_\eps)$ be a smooth solution to \eqref{eq1}-\eqref{ic} on $(0,T)$.
Then the energy identity
$$
  \frac{dE_\eps}{dt} + D_\eps = 0, \quad t\in(0,T),
$$
holds, where the energy $E_\eps$ and energy dissipation $D_\eps$ are defined
by, respectively,
$$
  E_\eps = \int_{\T^2}\left(\frac{1}{\eps^2}h(\rho_\eps)
  + \frac12\rho_\eps|u_\eps|^2
  + 2\eps^{2(\alpha-1)}|\na\sigma(\rho_\eps)|^2\right)dx, \quad
  D_\eps = 2\int_{\T^2}\mu(\rho_\eps)|D(u_\eps)|^2 dx.
$$
\end{lemma}

\begin{proof}
Multiply \eqref{eq1} by $\eps^{-2}h'(\rho_\eps)- \frac12|u_\eps|^2
- 2\eps^{2(\alpha-1)}\sigma'(\rho_\eps)\Delta\sigma(\rho_\eps)$,
integrate over $\T^2$, and then integrate by parts:
\begin{align*}
  0 &= \int_\T^2\Big(\frac{1}{\eps^{2}}\pa_t h(\rho_\eps)
  - \frac{1}{\eps^{2}}h''(\rho_\eps)\na\rho_\eps\cdot(\rho_\eps u_\eps)
  - \frac12|u_\eps|^2\pa_t\rho_\eps
  + \rho_\eps u_\eps\cdot\na u_\eps\cdot u_\eps  \\
  &\phantom{xx}{}+ 4\eps^{2(\alpha-1)}\na\sigma(\rho_\eps)\cdot
  \na\pa_t\sigma(\rho_\eps) - 2\eps^{2(\alpha-1)}\diver(\rho_\eps u_\eps)
  \sigma'(\rho_\eps)\Delta\sigma(\rho_\eps)\Big)dx.
\end{align*}
Multiplying \eqref{eq2} by $u_\eps$ and integrating over $\T^2$ gives, since
$u_\eps^\perp\cdot u_\eps=0$,
\begin{align*}
  0 &= \int_{\T^2}\Big(\pa_t(\rho u_\eps)\cdot u_\eps
  - \rho_\eps(u_\eps\otimes u_\eps):\na u_\eps
  + \frac{1}{\eps^2}\rho_\eps^{\gamma-1}\na\rho_\eps\cdot u_\eps \\
  &{}+ 2\eps^{2(\alpha-1)}
  \sigma'(\rho_\eps)\Delta\sigma(\rho_\eps)\diver(\rho_\eps u_\eps)
  - 2\mu(\rho_\eps) D(u_\eps):\na u_\eps\Big)dx,
\end{align*}
where ``:'' means summation over both matrix indices.
Observing that $h$ satisfies $h''(\rho_\eps)=\rho_\eps^{\gamma-2}$ and using the
identity $D(u_\eps):\na u_\eps=|D(u_\eps)|^2$, the sum of the
above two equations becomes
$$
  \frac{d}{dt}\int_{\T^2} \left(\frac{1}{\eps^2}h(\rho_\eps)
  + \frac12\rho_\eps|u_\eps|^2
   + 2\eps^{2(\alpha-1)}|\na\sigma(\rho_\eps)|^2\right)dx
   + 2\int_{\T^2}\mu(\rho_\eps)|D(u_\eps)|^2 dx = 0,
$$
which proves the lemma.
\end{proof}

A consequence of the energy identity is the following estimate.

\begin{lemma}
Let $(\rho_\eps,u_\eps)$ be a smooth solution to \eqref{eq1}-\eqref{ic} on $(0,T)$.
Then there exists $C>0$ such that for all $0<\eps<1$,
\begin{align}\label{est.rho}
  \|\rho_\eps-1\|_{L^\infty(0,T;L^\gamma(\T^2))}&\le C\eps^{\min\{1,2/\gamma\}} 
  \quad\mbox{if }\gamma>1, \\
  \|\rho_\eps-1\|_{L^\infty(0,T;L^2(\T^2))}&\le C\eps \quad\mbox{if }\gamma\ge 2. 
  \label{est.rho2}
\end{align}
\end{lemma}

\begin{proof}
If $\gamma=2$, $h(\rho)=\frac12(\rho-1)^2$, and the result follows
immediately from Lemma \ref{lem.ee}. Let $\gamma>2$.
We claim that $h(\rho)\ge |\rho-1|^\gamma/(\gamma(\gamma-1))$
for $\rho\ge 0$. Then the result follows again from the energy identity.
Indeed, the function $f(\rho)=\rho^\gamma-1-\gamma(\rho-1)-|\rho-1|^\gamma$
is convex in $(\frac12,\infty)$ and concave in $(0,\frac12)$. Since
the values $f(0)=\gamma-2$ and $f(\frac12)=\gamma/2-1$ are positive, $f\ge 0$
on $[0,\frac12]$. Furthermore, $f(1)=f'(1)=0$ which implies, together with
the convexity, that $f\ge 0$ in $[\frac12,\infty)$, proving the claim.
Finally, let $\gamma<2$. By \cite[p.~591]{LiMa98}, $h(\rho)\ge c_R|\rho-1|^2$
for $\rho\le R$ and $h(\rho)\ge c_R|\rho-1|^\gamma$ for $\rho>R$, for some $c_R>0$
and $R>0$. Hence, using H\"older's inequality and $\gamma<2$,
\begin{align*}
  \|\rho_\eps-1\|_{L^\gamma(\T^2)}^\gamma
  &\le C\left(\int_{\{\rho_\eps\le R\}}|\rho_\eps-1|^2 dx\right)^{\gamma/2}
  + \int_{\{\rho_\eps>R\}}|\rho_\eps-1|^\gamma dx \\
  &\le C\left(\int_{\{\rho_\eps\le R\}}h(\rho_\eps) dx\right)^{\gamma/2}
  + C\int_{\{\rho_\eps>R\}}h(\rho_\eps) dx
  \le C(\eps^\gamma+\eps^2)\le C\eps^\gamma,
\end{align*}
where here and in the following $C>0$ denotes a generic constant not depending
on $\eps$. Estimate \eqref{est.rho2} for $\gamma\ge 2$ follows from
$$
  \|\rho_\eps-1\|_{L^2(\T^2)}^2
  = \int_{\T^2}(\rho_\eps-1)^2 dx \le C\int_{\T^2}h(\rho_\eps)dx \le C\eps^2,
$$
which finishes the proof.
\end{proof}

We perform the formal limit $\eps\to 0$ in \eqref{eq1}-\eqref{eq2}. For this,
we observe that \eqref{eq1} can be written in terms of
$\phi_\eps=(\rho_\eps-1)/\eps$ as follows:
$$
  \pa_t \phi_\eps + \diver(\phi_\eps u_\eps) + \frac{1}{\eps}\diver u_\eps = 0.
$$
We apply the operator $\diver^\perp$
(defined by $\diver^\perp(v_1,v_2)=-\pa v_1/\pa x_2+\pa v_2/\pa x_1$)
to \eqref{eq2} and observe that $\diver^\perp(\rho_\eps u_\eps^\perp)/\eps
=\diver u_\eps/\eps + \diver(\phi_\eps u_\eps)=-\pa_t\phi_\eps$, by the
above equation. Then we find that
\begin{align}
  \pa_t\diver^\perp & (\rho_\eps u_\eps) + \diver^\perp\diver(\rho_\eps u_\eps\otimes
  u_\eps) - \pa_t\phi_\eps  \nonumber \\
  &= 2\eps^{2(\alpha-1)}\diver^\perp\big(\rho_\eps\na(\sigma'(\rho_\eps)
  \Delta\sigma(\rho_\eps))\big)
  + 2\diver^\perp\diver(\mu(\rho_\eps) D(u_\eps)). \label{phi.eps}
\end{align}
By the energy estimate, $\rho_\eps\to 1$ (in $L^\infty(0,T;L^\gamma(\T^2))$).
Assuming that $\phi_\eps\to \phi$ and $u_\eps\to\na^\perp\phi$ in suitable 
function spaces and employing the relations
$$
  \diver^\perp\diver(\na^\perp\phi\otimes\na^\perp\phi)
  = (\na^\perp\phi\cdot\na)(\Delta\phi), \quad
  2\diver^\perp\diver(D(\na^\perp\phi)) = \Delta^2\phi,
$$
the formal limit in \eqref{phi.eps} yields the limit equation \eqref{lim.eq}.
The initial condition reads as $\phi(\cdot,0)=\phi^0$, where
$\phi^0=\lim_{\eps\to 0}\phi_\eps(\cdot,0)$ in $\T^2$.
The energy and the energy dissipation of \eqref{lim.eq} equal
$$
  E_0 = \frac12\int_{\T^2}(|\na\phi|^2 + \phi^2)dx, \quad
  D_0 = 2\mu(1)\int_{\T^2}|D(\na^\perp\phi)|^2 dx.
$$
Multiplying the limiting equation by $\phi$ and using the properties
$$
  \int_{\T^2}(\na^\perp\phi\cdot\na)(\Delta\phi)\phi dx = 0, \quad
  \int_{\T^2}(\Delta\phi)^2 dx = 2\int_{\T^2}|D(\na^\perp\phi)|^2 dx,
$$
we find the energy identity of the viscous quasi-geostrophic equation:
$$
  \frac{dE_0}{dt} + D_0 = 0, \quad t>0.
$$

%%%%%%%%%%%%%%%%%%%%%%%%%%%%%%%%%%%%%%%%%%%%%%%%%%%%%%%%%%%%%%%

\section{Proof of Theorem \ref{thm}}\label{sec.thm}

First, we prove the following lemma.

\begin{lemma}\label{lem.Heps}
Let $T>0$, $\gamma>1$, and $0<\alpha<1$. Then
$$
  \lim_{\eps\to 0}H_\eps(t) = 0 \quad\mbox{uniformly in }(0,T),
$$
where $H_\eps$ is defined in \eqref{1.H}.
\end{lemma}

\begin{proof}
Using the definitions of the energy and energy dissipation as well as the
relation $\frac12 G_\eps(\phi_\eps)^2=\eps^{-2}h(\rho_\eps)$, we write
\begin{align*}
  H_\eps(t) &= (E_\eps + E)(t) + \int_0^t(D_\eps+D)(s)ds
  + \frac12\int_{\T^2}(\rho_\eps-1)|\na^\perp\phi|^2 dx \\
  &\phantom{xx}{}- \int_{\T^2}(G_\eps(\phi_\eps)-\phi_\eps)\phi dx
  -\int_{\T^2}\rho_\eps u_\eps\cdot\na^\perp\phi dx
  - \int_{\T^2}\phi_\eps\phi dx \\
  &\phantom{xx}{}+ 2\int_0^t\int_{\T^2}(\mu(\rho_\eps)-\mu(1))|D(\na^\perp\phi)|^2 dxds
  - 4\int_0^t\int_{\T^2}\mu(\rho_\eps) D(u_\eps):D(\na^\perp\phi) dxds \\
  &= I_1 + \cdots + I_8.
\end{align*}
The aim is to estimate $dH_\eps/dt$. To this end, we treat the integrals $I_j$
or their derivatives term by term.
By the energy estimates, $\frac{d}{dt}(I_1+I_2)=0$. The integral $I_3$ cancels with a contribution
originating from $I_5$; see below.
The estimate of $I_4,\ldots,I_8$ (or their derivatives) is performed in several steps.

{\em Step 1: estimate of $I_4$.} L'H\^opital's rule shows that for $\gamma>1$,
$$
  \lim_{z\to 0}\frac{h(1+z)}{z^2} = \frac12, \quad
  \lim_{z\to 0}\frac{1}{z}\left(\frac{h(1+z)}{z^2}-\frac12\right)
  = \frac{\gamma-2}{6}.
$$
Therefore, there exists a nonnegative function $f$, defined on $[0,\infty)$,
such that $h(1+z)=\frac12z^2 f(z)$ for $z\ge 0$, and a function $g$,
defined on $[0,\infty)$,
such that $f(z)-1=zg(z)$ for $z\ge 0$. Furthermore, the inequalities
$f(z)\ge f(0)=1$ and $|g(z)|\le C(1+z^{(\gamma-3)^+})$ hold, where $z^+=\max\{0,z\}$.
Finally, we claim that $f(z)=2h(1+z)/z^{2}\ge 2(1+z)^{\gamma-2}/(\gamma(\gamma-1))$
for $z\ge 0$ and $\gamma\geq 4$.
Indeed, the function $w(z)=h(1+z)-z^{2}(1+z)^{\gamma-2}/(\gamma(\gamma-1))$
is convex in $[0,\infty)$ and $w(0)=w'(0)=0$, which implies that $w(z)\ge 0$ in $[0,\infty)$, proving the claim.
With these preparations, we can estimate the difference $G_\eps(\phi_\eps)-\phi_\eps$
appearing in $I_4$:
\begin{align*}
  |G_\eps(\phi_\eps)-\phi_\eps|
  &= \left|\mbox{sign}(\phi_\eps)\left(\frac{\sqrt{2}}{\eps}
  \sqrt{h(1+\eps\phi_\eps)}-|\phi_\eps|\right)\right|
  = |\phi_\eps|\left|\sqrt{f(\eps\phi_\eps)}-1\right| \\
  &= \frac{|\phi_\eps|\,
  |f(\eps\phi_\eps)-1|}{\sqrt{f(\eps\phi_\eps)}+1}
  = \frac{|\phi_\eps|\,
  |\eps\phi_\eps|\,|g(\eps\phi_\eps)|}{\sqrt{f(\eps\phi_\eps)}+1}.
\end{align*}
In view of the bounds for $f$ and $g$ as well as the relation
$\eps\phi_\eps=\rho_\eps-1$, we infer that
\begin{equation}\label{est.G}
  |G_\eps(\phi_\eps)-\phi_\eps|
  \le \frac{C}{\eps}|\rho_\eps-1|^2\frac{1+\rho_\eps^{(\gamma-3)^+}}{{\sqrt{f(\eps\phi_\eps)}+1}}.
\end{equation}
This bound allows us to estimate $I_4$. Indeed, if $1<\gamma<4 $, by \eqref{est.rho},
$$
  I_4(t) \le \frac{C}{\eps}\|\phi\|_{L^\infty(0,T;L^\infty(\T^2))}
  \|\rho_\eps-1\|_{L^\infty(0,T;L^1(\T^2))}^2 \le C\eps^{2\min\{1,2/\gamma\}-1} = o(1)
$$
uniformly in $(0,T)$.
%, and if $3<\gamma<4$, using H\"older's inequality:
%\begin{align*}
%  I_4(t) &\le \frac{C}{\eps}\|\phi\|_{L^\infty(0,T;L^\infty(\T^2))}
%  \|\rho_\eps-1\|_{L^\infty(0,T;L^{2\gamma/3}(\T^2))}^2
%  \big(1+\|\rho_\eps\|^{\gamma-3}_{L^\infty(0,T;L^{\gamma}(\T^2))}\big) \\
%  &\le \frac{C}{\eps}\|\rho_\eps-1\|_{L^\infty(0,T;L^{\gamma}(\T^2))}^2
%  \le C\eps^{4/\gamma-1} = o(1).
%\end{align*}
Here and in the following, the constant $C>0$ depends on $\phi$ and its derivatives
but not on $\eps$.
If $\gamma\geq 4$, we have, using the upper bound of $f(z)$ 
for $\gamma\ge 4$, \eqref{est.G}, and $1+\eps\phi_\eps=\rho_\eps$,
$$
  |G_\eps(\phi_\eps)-\phi_\eps|
  \le \frac{C}{\eps}|\rho_\eps-1|^2
  \frac{1+\rho_\eps^{\gamma-3}}{C\rho_\eps^{(\gamma-2)/2}+1} 
  \le \frac{C}{\eps}|\rho_\eps-1|^2
  \big(1+\rho_\eps^{(\gamma-3)-(\gamma-2)/2}\big).
$$
We employ estimates \eqref{est.rho}-\eqref{est.rho2} and 
H\"older's inequality to conclude that
\begin{align*}
  I_4(t) &\le C\|\phi\|_{L^\infty(0,T;L^\infty(\T^2))}
  \eps^{-1}\|\rho_\eps-1\|_{L^\infty(0,T;L^2(\T^2))}
  \|\rho_\eps-1\|_{L^\infty(0,T;L^\gamma(\T^2))} \\
  &\phantom{xx}{}\times
  \big(1+\|\rho_\eps\|^{(\gamma-4)/2}_{L^\infty(0,T;L^{\gamma}(\T^2))}\big) \\
  &\le C\eps^{2/\gamma}\|\phi\|_{L^\infty(0,T;L^\infty(\T^2))} = o(1).
\end{align*}

{\em Step 2: estimate of $dI_5/dt$.} Inserting the momentum equation
\eqref{eq2} and integrating by parts, it follows that
\begin{align*}
  \frac{dI_5}{dt} &= -\int_{\T^2}\pa_t(\rho_\eps u_\eps)\cdot\na^\perp\phi dx
  -\int_{\T^2}\rho_\eps u_\eps\cdot\na^\perp\pa_t\phi dx \\
  &= -\int_{\T^2}\rho_\eps (u_\eps\otimes u_\eps):\na\na^\perp\phi dx
  + \frac{1}{\eps}\int_{\T^2}\rho_\eps u_\eps^\perp\cdot\na^\perp\phi dx \\
  &\phantom{xx}{}+ \frac{1}{\eps^2\gamma}\int_{\T^2}
  \na\rho_\eps^\gamma\cdot\na^\perp\phi dxd
  - 2\eps^{2(\alpha-1)}\int_{\T^2}
  \rho_\eps\na\big(\sigma'(\rho_\eps)\Delta\sigma(\rho_\eps)\big)
  \cdot\na^\perp\phi dx \\
  &\phantom{xx}{}+ 2\int_{\T^2}\mu(\rho_\eps) D(u_\eps):\na\na^\perp\phi dx
  - \int_{\T^2}\rho_\eps u_\eps\cdot\na^\perp\pa_t\phi dx \\
  &= J_1 +\cdots+ J_6.
\end{align*}
We treat the integrals $J_1,\ldots,J_6$ term by term. The integral $J_2$ can be
written as
$$
  J_2 = \frac{1}{\eps}\int_{\T^2}\rho_\eps u_\eps\cdot\na\phi dx.
$$
The third integral vanishes since $\diver\na^\perp=0$:
$$
  J_3 = -\frac{1}{\eps^2\gamma}\int_{\T^2}\rho_\eps^\gamma
  \diver(\na^\perp\phi)dx = 0.
$$
Using the identity \eqref{cons} and $\diver\na^\perp=0$, we compute
\begin{align*}
  J_4 &= \eps^{2(\alpha-1)}
  \int_{\T^2}\bigg(
  \left(\Delta S(\rho_\eps)-\frac12 S''(\rho_\eps)|\na\rho_\eps|^2\right)
  \diver(\na^\perp\phi) \\
  &\phantom{xx}{}- (\na\sigma(\rho_\eps)\otimes\na\sigma(\rho_\eps)):
  \na\na^\perp\phi\bigg)dx \\
  &\le CH_\eps.
\end{align*}
Integration by parts and using $\diver\na^\perp=0$ again yields
\begin{align*}
  J_5 &= -\int_{\T^2}\mu(\rho_\eps)u_\eps\cdot(\na^\perp\Delta\phi
  +\na\diver(\na^\perp\phi))dx \\
  &\phantom{xx}{}-\int_{\T^2}\mu'(\rho_\eps)(\na\rho_\eps\otimes u_\eps
  + u_\eps\otimes\na\rho_\eps):\na\na^\perp\phi dx \\
  &= -\int_{\T^2}\mu(\rho_\eps)u_\eps\cdot\na^\perp\Delta\phi dx \\
  &\phantom{xx}{}- 2\int_{\T^2}
  \frac{\mu'(\rho_\eps)}{\sqrt{\rho_\eps}\sigma'(\rho_\eps)}
  \big(\na\sigma(\rho_\eps)\otimes(\sqrt{\rho_\eps}u_\eps)
  + (\sqrt{\rho_\eps}u_\eps)\otimes\na\sigma(\rho_\eps)\big):\na\na^\perp\phi dx.
\end{align*}
The assumptions on $\mu$ and $\sigma$ (see \eqref{sm}) yield
$\mu'(\rho_\eps)/(\sqrt{\rho_\eps}\sigma'(\rho_\eps))=\rho_{\eps}^{m-s-1/2}$.
Hence, applying the Cauchy-Schwarz inequality, the last integral is bounded
from above by
$$
  C\|\na\sigma(\rho_\eps)\|_{L^2(\T^2)}\|\sqrt{\rho_\eps}u_\eps\|_{L^2(\T^2)}
  \le C\eps^{2(1-\alpha)} = o(1).
$$
We conclude that
$$
  J_5 \le -\int_{\T^2}\mu(\rho_\eps)u_\eps\cdot\na^\perp\Delta\phi dx + o(1).
$$
The integral $J_6$ remains unchanged. Finally, we estimate $J_1$. To this end,
we add and substract the expression $\na^\perp\phi$ such that
$J_1=K_1+\cdots+K_4$, where
\begin{align*}
  K_1 &= -\int_{\T^2}\rho_\eps(u_\eps-\na^\perp\phi)\otimes
  (u_\eps-\na^\perp\phi):\na\na^\perp\phi dx, \\
  K_2 &= -\int_{\T^2}\rho_\eps\na^\perp\phi\otimes
  u_\eps:\na\na^\perp\phi dx, \\
  K_3 &= -\int_{\T^2}\rho_\eps u_\eps\otimes
  \na^\perp\phi:\na\na^\perp\phi dx, \\
  K_4 &= \int_{\T^2}\rho_\eps\na^\perp\phi\otimes
  \na^\perp\phi:\na\na^\perp\phi dx.
\end{align*}
The first integral can be bounded by the modulated energy:
$$
  K_1 \le C\int_{\T^2}\rho_\eps|u_\eps-\na^\perp\phi|^2 dx
  \le C H_\eps.
$$
A reformulation yields
$$
  K_2 = -\int_{\T^2}\rho_\eps u_{\eps}\cdot
  \big((\na^\perp\phi\cdot\na)\na^\perp\phi\big) dx.
$$
We employ the continuity equation \eqref{eq1} to find
\begin{align*}
  K_3 &= -\frac12\int_{\T^2}\rho_\eps u_\eps\cdot\na|\na^\perp\phi|^2 dxd
  = \frac12\int_{\T^2}\diver(\rho_\eps u_\eps)|\na^\perp\phi|^2 dx \\
  &= -\frac12\int_{\T^2}\pa_t(\rho_\eps-1)|\na^\perp\phi|^2 dx \\
  &= -\frac12\frac{d}{dt}\int_{\T^2}(\rho_\eps-1)|\na^\perp\phi|^2 dx
  + \frac12\int_{\T^2}(\rho_\eps-1)\pa_t|\na^\perp\phi|^2 dx \\
  &= -\frac{dI_3}{dt} + o(1).
\end{align*}
Finally, using again $\diver\na^\perp=0$,
\begin{align*}
  K_4 &= -\int_{\T^2}\rho_\eps
  \big((\na^\perp\phi\cdot\na)\na^\perp\phi\big)\cdot\na^\perp\phi dx \\
  &= -\int_{\T^2}(\rho_\eps-1)
  \big((\na^\perp\phi\cdot\na)\na^\perp\phi\big)\cdot\na^\perp\phi dx
  - \int_{\T^2}
  \big((\na^\perp\phi\cdot\na)\na^\perp\phi\big)\cdot\na^\perp\phi dx \\
  &= -\int_{\T^2}(\rho_\eps-1)
  \big((\na^\perp\phi\cdot\na)\na^\perp\phi\big)\cdot\na^\perp\phi dx
  - \frac12\int_{\T^2}\na^\perp\phi\cdot\na(|\na^\perp\phi|^2) dx \\
  &= -\int_{\T^2}(\rho_\eps-1)
  \big((\na^\perp\phi\cdot\na)\na^\perp\phi\big)\cdot\na^\perp\phi dx
  + \frac12\int_{\T^2}\diver(\na^\perp\phi)|\na^\perp\phi|^2 dx \\
  &= o(1).
\end{align*}
In the last step, we have employed estimate \eqref{est.rho} for $\rho_\eps-1$.
Summarizing the estimates for $K_1,\ldots,K_4$, we have shown that
$$
  J_1 \le CH_\eps - \frac{dI_3}{dt}
  - \int_{\T^2}\big((\na^\perp\phi\cdot\na)\na^\perp\phi\big)
  \cdot(\rho_\eps u_{\eps}) dx + o(1).
$$
Then, summarizing the estimates for $J_1,\ldots,J_6$, we obtain
\begin{align*}
  \frac{dI_5}{dt} &\le CH_\eps - \frac{dI_3}{dt}
  + \frac{1}{\eps}\int_{\T^2}\rho_\eps u_\eps\cdot\na\phi dx \\
  &\phantom{xx}{}- \int_{\T^2}
  \big((\pa_t+\na^\perp\phi\cdot\na)\na^\perp\phi
  + \mu(1)\na^\perp\Delta\phi\big)\cdot(\rho_\eps u_\eps)dx \\
  &\phantom{xx}{}- \int_{\T^2}
  \big(\mu(\rho_\eps)-\mu(1)\rho_\eps\big)u_\eps\cdot\na^\perp\Delta\phi dx
  + o(1).
\end{align*}
The last integral can be estimated by employing the assumptions on $\mu$
and H\"older's inequality:
$$
  \int_{\T^2}\frac{\mu(\rho_\eps)-\mu(1)\rho_\eps}{\sqrt{\rho_\eps}}
  \sqrt{\rho_\eps}u_\eps\cdot\na^\perp\Delta\phi dx
  \le C\|\rho_\eps^{m-1/2}-\rho_\eps^{1/2}\|_{L^2(\T^2)}
  \|\sqrt{\rho_\eps}u_\eps\|_{L^2(\T^2)}.
$$
We claim that the first factor on the right-hand side is of order $o(1)$.
To prove this statement, we consider first $\frac12<m<1$:
$$
  \|\rho_\eps^{m-1/2}-\rho_\eps^{1/2}\|_{L^2(\T^2)}^2
  \le \int_{\T^2}\rho_\eps^{2m-1}|\rho_\eps-1|^{2(1-m)} dx
  \le \|\rho_\eps\|_{L^\gamma(\T^2)}^{2m-1}\|\rho_\eps-1\|_{L^p(\T^2)}^{2(1-m)},
$$
where $p=2\gamma(1-m)/(\gamma-2m+1)$. The inequality $p\le\gamma$ is equivalent to
$\gamma\ge 1$. Note that the H\"older inequality can be applied since we supposed
that $2m-1\le\gamma$; see \eqref{sm}.
Second, let $1<m\le 2$ (the case $m=1$ being trivial). We compute
$$
  \|\rho_\eps^{m-1/2}-\rho_\eps^{1/2}\|_{L^2(\T^2)}^2
  \le \int_{\T^2}\rho_\eps|\rho_\eps-1|^{2(m-1)}dx
  \le \|\rho_\eps\|_{L^\gamma(\T^2)}\|\rho_\eps-1\|_{L^q(\T^2)}^{2(m-1)},
$$
where $q=2\gamma(m-1)/(\gamma-1)$, and $q\le\gamma$ if and only if $m\le(\gamma+1)/2$.
Finally, if $2\le m\le(\gamma+1)/2$, we find that
$$
  \|\rho_\eps^{m-1/2}-\rho_\eps^{1/2}\|_{L^2(\T^2)}^2
  \le C\int_{\T^2}\rho_\eps(1+\rho_\eps^{m-2})^2|\rho_\eps-1|^{2}dx
  \le C(1+\|\rho_\eps\|_{L^\gamma(\T^2)}^{2m-3})\|\rho_\eps-1\|_{L^r(\T^2)}^2,
$$
with $r=2\gamma/(\gamma-2m+3)$ satisfying $r\le\gamma$ if and only if
$m\le(\gamma+1)/2$. We conclude that
$$
  \int_{\T^2}\frac{\mu(\rho_\eps)-\mu(1)\rho_\eps}{\sqrt{\rho_\eps}}
  \sqrt{\rho_\eps}u_\eps\cdot\na^\perp\Delta\phi dx
  \le C\|\rho_\eps-1\|_{L^\gamma(\T^2)}^\beta
$$
for some $\beta>0$, and together with \eqref{est.rho}, this shows that the
integral is of order $o(1)$. Therefore,
\begin{align}
  \frac{dI_5}{dt} &\le CH_\eps - \frac{dI_3}{dt}
  + \frac{1}{\eps}\int_{\T^2}\rho_\eps u_\eps\cdot\na\phi dx \nonumber \\
  &\phantom{xx}{}- \int_{\T^2}
  \big((\pa_t+\na^\perp\phi\cdot\na)\na^\perp\phi
  + \mu(1)\na^\perp\Delta\phi\big)\cdot(\rho_\eps u_\eps)dx + o(1). \label{I5}
\end{align}

{\em Step 3: estimate of $dI_6/dt$.}
Employing \eqref{eq1} and \eqref{lim.eq}, we can write
\begin{align}
  \frac{dI_6}{dt} &= -\int_{\T^2}\pa_t\phi_\eps\phi dx
  - \int_{\T^2}\phi_\eps\pa_t\phi dx \nonumber \\
  &= \frac{1}{\eps}\int_{\T^2}\diver(\rho_\eps u_\eps)\phi dx
  - \int_{\T^2}\big((\pa_t+\na^\perp\phi\cdot\na)(\Delta\phi)
  -\mu(1)\Delta^2\phi\big)\phi_\eps dx \nonumber \\
  &= -\frac{1}{\eps}\int_{\T^2}\rho_\eps u_\eps\cdot\na\phi dx
  + \int_{\T^2}\big((\pa_t+\na^\perp\phi\cdot\na)(\na^\perp\phi)
  -\mu(1)\na^\perp\Delta\phi\big)\cdot\na^\perp\phi_\eps dx. \label{na.phi}
\end{align}
We observe that the first integral on the right-hand side
cancels with the corresponding integral in \eqref{I5}. To deal with
the second integral, we employ again the momentum equation \eqref{eq2}.
We write
$$
  \frac{1}{\gamma}\na\rho_\eps^\gamma
  = (\gamma-1)\na h(\rho_\eps) + \na(\rho_\eps-1)
  = (\gamma-1)\na h(\rho_\eps) + \eps\na\phi_\eps.
$$
Then, because of $(u_\eps^\perp)^\perp=-u_\eps$, \eqref{eq2} is equivalent to
$$
  \na^\perp\phi_\eps = \rho_\eps u_\eps - \eps F_\eps^\perp,
$$
where
\begin{align*}
  F_\eps &= \pa_t(\rho_\eps u_\eps) + \diver(\rho_\eps u_\eps\otimes u_\eps)
  + \frac{\gamma-1}{\eps^2}\na h(\rho_\eps)
  - 2\diver(\mu(\rho_\eps) D(u_\eps)) \\
  &\phantom{xx}{}
  - \eps^{2(\alpha-1)}\Big(\na\Delta S(\rho_\eps)
  - \frac12\na(S''(\rho_\eps)|\rho_\eps|^2)
  - \diver\big(\na\sigma(\rho_\eps)\otimes\na\sigma(\rho_\eps)\big)\Big).
\end{align*}
Replacing $\na^\perp\phi_\eps$ in the second integral in \eqref{na.phi}
by the above expression gives
\begin{align*}
  \int_{\T^2} & \big((\pa_t+\na^\perp\phi\cdot\na)(\na^\perp\phi)
  -\mu(1)\na^\perp\Delta\phi\big)\cdot\na^\perp\phi_\eps dx \\
  &= \int_{\T^2}
  \big((\pa_t+\na^\perp\phi\cdot\na)(\na^\perp\phi)
  -\mu(1)\na^\perp\Delta\phi\big)\cdot(\rho_\eps u_\eps-\eps F_\eps^\perp) dx.
\end{align*}

We claim that the integral containing $F_\eps^\perp$ is bounded in an appropriate
space. Indeed, let $\psi$ be a smooth (vector-valued) test function.
The first term of $F_\eps$ is written in weak form as follows:
\begin{align*}
  \int_0^T\int_{\T^2}\pa_t(\rho_\eps u_\eps)\cdot\psi dxds
  &= -\int_0^T\int_{\T^2}\rho_\eps u_\eps \cdot\pa_t\psi dxds
  + \int_{\T^2}(\rho_\eps u_\eps)(t)\cdot\psi(t)dx \\
  &\phantom{xx}{}- \int_{\T^2}\rho_\eps^0 u_\eps^0 \cdot\psi(0)dx.
\end{align*}
These integrals are bounded if $\rho_\eps u_\eps$ is bounded in
$L^\infty(0,T;L^1(\T^2))$. This is the case, since mass conservation and the
energy estimate show that
$$
  \int_{\T^2}|\rho_\eps u_\eps| dx
  \le \frac12\int_{\T^2}\rho_\eps dx + \frac12\int_{\T^2}\rho_\eps|u_\eps|^2 dx
$$
is uniformly bounded in $(0,T)$. An integration by parts gives
$$
  \int_0^T\int_{\T^2}\diver(\rho_\eps u_\eps\otimes u_\eps)\cdot\psi dxds
  = -\int_0^T\int_{\T^2}\rho_\eps u_\eps\otimes u_\eps:\na\psi dxds,
$$
and this integral is uniformly bounded, by the energy estimate.
Furthermore, again integrating by parts,
\begin{align*}
  \int_0^T\int_{\T^2} & \left(\frac{\gamma-1}{\eps^2}\na h(\rho_\eps)
  - 2\diver(\mu(\rho_\eps) D(u_\eps))\right)\cdot\psi dxds \\
  &= -\int_0^T\int_{\T^2}\left(\frac{\gamma-1}{\eps^2}h(\rho_\eps){\mathbb I}
  - 2\mu(\rho_\eps) D(u_\eps)\right):\na\psi dxds,
\end{align*}
which is uniformly bounded since we can estimate
$$
  \int_0^T\int_{\T^2}|\mu(\rho_\eps) D(u_\eps)|dxds
  \le \frac12\int_0^T\int_{\T^2}\mu(\rho_\eps) dxds
  + \frac12\int_0^T\int_{\T^2}\mu(\rho_\eps)|D(u_\eps)|^2 dxds
$$
and $\mu(\rho_\eps)\le C(1+\rho_\eps^\gamma)$.
Also the remaining terms are bounded since
\begin{align*}
  \eps^{2(\alpha-1)} & \int_0^T\int_{\T^2}
  \Big(\na\Delta(S(\rho_\eps)-S(1))-\frac12\na(S''(\rho_\eps)|\na\rho_\eps|^2)
  - \diver(\na\sigma(\rho_\eps)\otimes\na\sigma(\rho_\eps))\Big)\cdot\psi dxds \\
  &= -\eps^{2(\alpha-1)}\int_0^T\int_{\T^2}
  \big((S(\rho_\eps)-S(1))\Delta\diver\psi
  + \frac12 S''(\rho_\eps)|\na\rho_\eps|^2\diver\psi \\
  &\phantom{xx}{}
  - (\na\sigma(\rho_\eps)\otimes\na\sigma(\rho_\eps)):\na\psi\big)dxds.
\end{align*}
Using the H\"older continuity of $S(z)=(s/2) z^{2s}$, $z\ge 0$,
the first summand can be estimated by $C|\rho_\eps-1|^{\min\{1,2s\}}$.
We infer that the corresponding integral is of order $o(1)$.
We formulate the second summand as
$$
  \frac12 \eps^{2(\alpha-1)}(2s-1)\int_0^t\int_{\T^2}|\na\sigma(\rho_\eps)|^2
  \diver\psi dxds.
$$
In view of the energy estimate, this integral as well as the third summand
are uniformly bounded. This shows that
\begin{align*}
  \int_{\T^2} & \big((\pa_t+\na^\perp\phi\cdot\na)(\na^\perp\phi)
  -\mu(1)\na^\perp\Delta\phi\big)\cdot\na^\perp\phi_\eps dx \\
  &= \int_{\T^2}
  \big((\pa_t+\na^\perp\phi\cdot\na)(\na^\perp\phi)
  -\mu(1)\na^\perp\Delta\phi\big)\cdot(\rho_\eps u_\eps) dx + o(1),
\end{align*}
and consequently, \eqref{na.phi} becomes
\begin{align*}
  \frac{dI_6}{dt} &= -\frac{1}{\eps}\int_{\T^2}\rho_\eps u_\eps\cdot
  \na\phi dx \\
  &\phantom{xx}{}+\int_{\T^2}
  \big((\pa_t+\na^\perp\phi\cdot\na)(\na^\perp\phi)
  -\mu(1)\na^\perp\Delta\phi\big)\cdot(\rho_\eps u_\eps) dx + o(1).
\end{align*}

{\em Step 4: estimate of $dI_7/dt$.}
The function $\mu$ satisfies
$|\mu(z)-\mu(1)|=|z^m-1|\le |z-1|^m$ if $m\le 1$ and
$|\mu(z)-\mu(1)|\le C(1+z^{m-1})|z-1|$ if $m>1$, for $z\ge 0$.
Therefore, if $m\le 1$, taking into account \eqref{est.rho},
$$
  \frac{dI_7}{dt} \le 2\|\rho_\eps-1\|_{L^\infty(0,T;L^\gamma(\T^2))}^m
  \|D(\na^\perp\phi)\|_{L^\infty(0,T;L^{2\gamma/(\gamma-m)}(\T^2))}^2
  \le C\eps^{m\min\{1,2/\gamma\}}.
$$
Moreover, if $1<m\le(\gamma+1)/2$, using H\"older's inequality,
$$
   \frac{dI_7}{dt}
   \le C\big(1+\|\rho_\eps\|_{L^\infty(0,T;L^{(m-1)\gamma/(\gamma-1)}(\T^2))}\big)
   \|\rho_\eps-1\|_{L^\infty(0,T;L^\gamma(\T^2))} \le C\eps^{\min\{1,2/\gamma\}}.
$$
The norm of $\rho_\eps$ is uniformly bounded since
$(m-1)\gamma/(\gamma-1)\le\gamma$ is equivalent to $m\le\gamma$.

{\em Step 5: estimate of $dI_8/dt$.}
Integration by parts yields
\begin{align*}
  \frac{dI_8}{dt}
  &= \int_{\T^2}\mu'(\rho_\eps)\na\rho_\eps\otimes u_\eps:\na\na^\perp\phi dx
  + 2\int_{\T^2}\mu(\rho_\eps)u_\eps\cdot\na^\perp\Delta\phi dx \\
  &= \int_{\T^2}µ\frac{\mu'(\rho_\eps)}{\sqrt{\rho_\eps}\sigma'(\rho_\eps)}
  \na\sigma(\rho_\eps)\otimes(\sqrt{\rho_\eps}u_\eps):\na\na^\perp\phi dx \\
  &\phantom{xx}{}
  + 2\int_{\T^2}(\mu(\rho_\eps)-\mu(1)\rho_\eps)u_\eps\cdot\na^\perp\Delta\phi dx
  + 2\mu(1)\int_{\T^2}\rho_\eps u_\eps\cdot\na^\perp\Delta\phi dx.
\end{align*}
By definition of $\mu$ and $\sigma$ (see \eqref{sm}), it follows that
\begin{align*}
  \frac{dI_8}{dt}
  &\le C\|\na\sigma(\rho_\eps)\|_{L^2(\T^2)}\|\sqrt{\rho_\eps}u_\eps\|_{L^2(\T^2)}
  + C\|\rho_\eps^{m-1/2}-\rho_\eps^{1/2}\|_{L^2(\T^2)}
  \|\sqrt{\rho_\eps}u_\eps\|_{L^2(\T^2)} \\
  &\phantom{xx}{}+ 2\mu(1)\int_{\T^2}\rho_\eps u_\eps\cdot\na^\perp\Delta\phi dx.
\end{align*}
Because of the energy estimate, the first summand is of order $o(1)$.
The second summand has been estimated in Step 2, and it has been found that
it is also of order $o(1)$. This shows that
$$
  \frac{dI_8}{dt} \le 2\mu(1)\int_{\T^2}\rho_\eps u_\eps\cdot\na^\perp\Delta\psi dx
  + o(1).
$$

{\em Step 6: conclusion.}
Adding the estimates for $dI_4/dt,\ldots,dI_8/dt$, most of the integrals cancel, and
we end up with
$$
  \frac{dH_\eps}{dt} \le CH_\eps + \frac{dI_4}{dt} + o(1).
$$
Integrating over $(0,t)$ gives
$$
  H_\eps(t) \le H_\eps(0) + C\int_0^t H_\eps(s)ds + I_4(t) - I_4(0) + o(1).
$$
By Step 1, $I_4(t)=o(1)$. Furthermore, $I_4(0)=o(1)$ by assumption.
It holds that $H_\eps(0)=o(1)$ since
\begin{align*}
  \|\sqrt{\rho_\eps^0}(u_\eps^0-\na^\perp\phi^0)\|_{L^2(\T^2)}
  &\le \|\sqrt{\rho_\eps^0}u_\eps^0 - \na^\perp\phi^0\|_{L^2(\T^2)}
  + \|(1-\sqrt{\rho_\eps^0})\na^\perp\phi^0\|_{L^2(\T^2)} \\
  &\le \|\sqrt{\rho_\eps^0}u_\eps^0 - \na^\perp\phi^0\|_{L^2(\T^2)}
  + \|1-\rho_\eps^0\|_{L^2(\T^2)}\|\na^\perp\phi^0\|_{L^\infty(\T^2)} \\
  &= o(1)
\end{align*}
and since the initial data are well prepared.
Then the Gronwall lemma implies that $H_\eps(t)=o(1)$ finishing the proof.
\end{proof}

We are now in the position to prove Theorem \ref{thm} which
is a consequence of Lemma \ref{lem.Heps}. 
We observe that
by \eqref{est.rho}, $\rho_\eps\to 1$ in $L^\infty(0,T;L^\gamma(\T^2))$ and,
using the H\"older inequality and $2\gamma/(\gamma+1)<\gamma$,
\begin{align}
  \|\rho_\eps u_\eps - \na^\perp\phi\|_{L^\infty(0,T;L^{2\gamma/(\gamma+1)}(\T^2))}
  &\le \|\sqrt{\rho_\eps}\|_{L^\infty(0,T;L^{2\gamma}(\T^2))}
  \|\sqrt{\rho_\eps}(u_\eps-\na^\perp\phi)\|_{L^\infty(0,T;L^2(\T^2))}\nonumber \\
  &\phantom{xx}{}+ \|\rho_\eps-1\|_{L^\infty(0,T;L^{2\gamma/(\gamma+1)}(\T^2))}
  \|\na^\perp\phi\|_{L^\infty(0,T;L^\infty(\T^2))} \nonumber\\
  &\le C\|\sqrt{\rho_\eps}(u_\eps-\na^\perp\phi)\|_{L^\infty(0,T;L^2(\T^2))} 
  \label{na.aab} \\
  &\phantom{xx}{}+ C \|\rho_\eps-1\|_{L^\infty(0,T;L^{\gamma}(\T^2))}\nonumber.
\end{align}
We conclude that $\rho_\eps u_\eps\to \na^\perp\phi$ in
$L^\infty(0,T;L^{2\gamma/(\gamma+1)}(\T^2))$. 

Next, let $\gamma\ge 2(1-s)$ and $0<s<1/2$. Because of assumption \eqref{sm}, i.e.\
$\gamma\ge 2s$, we have $2\gamma/(\gamma+2(1-s))\le\gamma$, and hence,
$$
  \rho_\eps \to 1 \quad\mbox{in }L^\infty(0,T;L^{2\gamma/(\gamma+2(1-s))}(\T^2))
$$
as $\eps\to 0$. Furthermore, since $\alpha<1$,
$\na\sigma(\rho_\eps)\to 0$ in $L^\infty(0,T;L^2(\T^2))$
as $\eps\to 0$ and thus, by H\"older's inequality,
\begin{align}
  \|\na(\rho_\eps-1)\|_{L^\infty(0,T;L^{2\gamma/(\gamma+2(1-s))}(\T^2))}
  &= \|\sigma'(\rho_\eps)^{-1}\na\sigma(\rho_\eps)
  \|_{L^\infty(0,T;L^{2\gamma/(\gamma+2(1-s))}(\T^2))}\label{na.aaa} \\
  &\le \|\rho_\eps\|_{L^\infty(0,T;L^\gamma(\T^2))}^{1-s}
  \|\na\sigma(\rho_\eps)\|_{L^\infty(0,T;L^2(\T^2))} \to 0. \nonumber
\end{align}
We infer that $\rho_\eps\to 1$ in
$L^\infty(0,T;W^{1,2\gamma/(\gamma+2(1-s))}(\T^2))$.
Because of the continuous embedding 
$W^{1,2\gamma/(\gamma+2(1-s))}(\T^2)\hookrightarrow
L^{\gamma/(1-s)}(\T^2)$, this implies that $\rho_\eps\to 1$ in $L^\infty(0,T;L^{\gamma/(1-s)}(\T^2))$. Since 
$2\gamma/(\gamma+2(1-s)^2)\le \gamma/(1-s)$, 
this gives $\rho_\eps\to 1$ in $L^\infty(0,T;L^{2\gamma/(\gamma+2(1-s)^2)}(\T^2))$.
Applying the same procedure as in \eqref{na.aaa} again, we obtain
$$
  \|\na(\rho_\eps-1)\|_{L^\infty(0,T;L^{2\gamma/(\gamma+2(1-s)^2)}(\T^2))}
  \le \|\rho_\eps\|_{L^\infty(0,T;L^{\gamma/(1-s)}(\T^2))}^{1-s}
  \|\na\sigma(\rho_\eps)\|_{L^\infty(0,T;L^2(\T^2))}\to 0.
$$
Hence, $\rho_\eps\to 1$ in $L^\infty(0,T;W^{1,2\gamma/(\gamma+2(1-s)^2)}(\T^2))$
and in $L^\infty(0,T;L^{\gamma/(1-s)^2}(\T^2))$. Repeating this argument,
we conclude that $\rho^{\eps}\to 1$ in $L^\infty(0,T;L^{p}(\T^2))$ for all $p<\infty$.

For the momentum, we obtain for $p\ge 1$
\begin{align*}
  \|\rho_\eps u_\eps - \na^\perp\phi\|_{L^\infty(0,T;L^{2p/(p+1)}(\T^2))}
  &\le \|\sqrt{\rho_\eps}\|_{L^\infty(0,T;L^{2p}(\T^2))}
  \|\sqrt{\rho_\eps}(u_\eps-\na^\perp\phi)\|_{L^\infty(0,T;L^2(\T^2))} \\
  &\phantom{xx}{}+ \|\rho_\eps-1\|_{L^\infty(0,T;L^{2p/(p+1)}(\T^2))}
  \|\na^\perp\phi\|_{L^\infty(0,T;L^\infty(\T^2))} \\
  &\le C\|\sqrt{\rho_\eps}(u_\eps-\na^\perp\phi)\|_{L^\infty(0,T;L^2(\T^2))} \\
  &\phantom{xx}{}+ C \|\rho_\eps-1\|_{L^\infty(0,T;L^{p}(\T^2))}.
\end{align*}
This shows that $\rho_\eps u_\eps \to \na^\perp\phi$ in
$L^\infty(0,T;L^q(\T^2))$ for all $q<2$.

Finally, let $\gamma\ge 2$ and $s=1$. Then $\rho_\eps\to 1$ in $L^\infty(0,T;H^1(\T^2))$
and, by the continuous embedding $H^1(\T^2)\hookrightarrow L^p(\T^2)$ for
all $p<\infty$, also $\rho_\eps\to 1$ in $L^\infty(0,T;L^p(\T^2))$ for all
$p<\infty$. The theorem is proved.

%%%%%%%%%%%%%%%%%%%%%%%%%%%%%%%%%%%%%%%%%%%%%%%%%%%%%%%%%%%%%%%%%%%%%%%%%%%%%%%%%%%

\begin{appendix}
\section{Local existence of smooth solutions}\label{sec.loc}

The local existence of smooth solutions to the Navier-Stokes-Korteweg system 
\eqref{eq1}-\eqref{eq2} can be shown similarly as in \cite{LiMa04}. We only sketch
the proof since it is highly technical and does not involve new ideas.
First, we rewrite \eqref{eq1}-\eqref{eq2}, setting $\rho=\rho_\eps$, $u=u_\eps$, and
$\eps=1$. Taking the divergence of \eqref{eq2}
and replacing $\diver\pa_t(\rho u)$ by \eqref{eq1}, which has been differentiated
with respect to time, we obtain
\begin{align*}
  \pa_{tt}^2\rho -\frac{1}{\gamma}\Delta\rho^\gamma + 2\rho\sigma'(\rho)^2\Delta^2\rho 
  &= -\diver\diver(\rho u\otimes u) - \diver(\rho u^\perp) \\
  &\phantom{xx}{}+ 2\diver\diver(\mu(\rho)D(u)) + F[\rho],
\end{align*}
where $F[\rho]=2\diver(\rho\na(\sigma'(\rho)\Delta\sigma(\rho)))
-2\rho\sigma'(\rho)^2\Delta^2\rho$ involves only three derivatives.
This formulation allows one to treat the momentum equation as a nonlinear
fourth-order wave equation for which existence and regularity results can be 
applied. In order to derive some regularity for the velocity, 
Li and Marcati \cite{LiMa04} assumed that $\operatorname{curl} u=0$.
Then $u$ is reconstructed from the problem
$$
  \diver v = -\frac{1}{\rho}(\pa_t\rho + \na\rho\cdot u),
  \quad \operatorname{curl}v=0, \quad \int_{\T^2}v(t)dx = \bar u(t).
$$
Theorem 2.1 in \cite{LiMa04} gives the existence of a unique solution
$u\in H^{s+1}(\T^2)$ to this problem, provided that the right-hand side satisfies
$-(\pa_t\rho + \na\rho\cdot u)/\rho\in H^s(\T^2)$.
Actually, Li and Marcati replace the right-hand side by 
$-(\pa_t\rho + \na\rho\cdot u)/\psi$, where $\psi$ solves the mass equation
$$
  \pa_t\psi + \psi\diver v + u\cdot\na \rho = 0, \quad t>0, \quad \psi(0)=\rho^0.
$$
The reason is that this equation can be solved explicitly, 
yielding strictly positive solutions $\psi$. 
The existence proof is based on an iteration scheme: Given $(\rho_p,\psi_p,u_p,v_p)$,
solve
\begin{align*}
  & \diver v_{p+1}=f_p(t), \quad \operatorname{curl}v_{p+1}=0, \quad
  \int_{\T^2}v_{p+1}(t)dx=\bar u(t), \\
  & \pa_t\psi_{p+1} + \psi_{p+1}\diver v_p + u_p\cdot\na \rho_p = 0, 
  \quad t>0, \quad \psi(0)=\rho^0, \\
  & \pa_{tt}^2\rho_{p+1} - \frac{1}{\gamma}\Delta\rho_{p+1}^\gamma
  + \psi_p\sigma'(\psi_p)^2\Delta^2\rho_{p+1} = g_p(t), \quad t>0, \\
  & \rho_{p+1}(0)=\rho^0, \ \pa_t\rho_{p+1}(0)=-\rho^0\diver u^0-\na\rho^0\cdot u^0, \\
  & \pa_t u_{p+1} + u_{p+1}^\perp = h_p(t),
\end{align*}
where $f_p(t)$, $g_p(t)$, and $h_p(t)$ contain the remaining terms
(see \cite[Section 3]{LiMa04} for details). The existence of solutions to
these linear problems follows from ODE theory and the theory of wave equations.
The main effort is now to derive uniform estimates
in Sobolev spaces $H^k(\T^2)$. 
This is done by multiplying the above equations by suitable test functions
and assuming that $T>0$ is sufficiently small. By compactness, there exists
a subsequence of $(\rho_p,\psi_p,u_p,v_p)$ which converges in a suitable
Sobolev space to $(\rho,\psi,u,v)$ as $p\to\infty$. 
This limit allows us also to show that $\rho=\psi\ge 0$ and $u=v$.
This shows the existence of local smooth solutions under the assumption
of irrotational flow $\operatorname{curl}u=0$.

Next, we prove the existence of local smooth solutions to the
quasi-geostrophic equation \eqref{lim.eq}. We set $\mu:=\mu(1)>0$.

\begin{theorem}[Local existence for the quasi-geostrophic equation]
Let $\phi_0\in C^\infty(\T^2)$. Then there exists $T>0$ and a smooth 
solution $\phi$ to \eqref{lim.eq}-\eqref{lim.ic} for $0\le t\le T$.
\end{theorem}

\begin{proof}
The idea of the proof is to apply the theory of linear semigroups. 
Let $p>2$ and let $A_p:W^{2,p}(\T^2)\to\R$, $A_p(u)=-\mu\Delta u+u$. Then 
$A_p$ is a sectorial operator satisfying $\Re(\lambda)=1$ for all
$\lambda\in\sigma(A_p)$, where $\sigma(A_p)$ denotes the spectrum of
$A_p$. Consequently, $A_p$ possesses the fractional powers $A_p^\beta$
for $\beta\ge 0$, defined on the domain $X^{\beta,p}=D(A_p^\beta)$.
This space, endowed with its graph norm, satisfies 
$X^{\beta,p}\hookrightarrow W^{k,q}(\T^2)$ if $k-2/q<2\beta-2/p$, $q\ge p$
\cite[Theorem 1.6.1]{Hen81}.
Let $\max\{1-1/p,1/2+1/(2p)\}<\beta<1$ and set $X:=X^{\beta,p}$.
The operator $A_p$ generates an analytical semigroup $e^{-t A_p}$ ($t\ge 0$)
\cite[Theorem 1.3.4]{Hen81},
and the following estimates hold for all $t>0$ \cite[Theorem 1.4.3]{Hen81}:
\begin{align*}
  \|A_p e^{-tA_p}u\|_{L^p(\T^2)} 
  &\le Ct^{-\beta}e^{-\delta t}\|u\|_{L^p(\T^2)}, \\
  \|(e^{-tA_p}-I)v\|_{L^p(\T^2)} &\le Ct^{\beta}\|A_pv\|_{L^p(\T^2)}
  \le Ct^{\beta}\|v\|_{X}
\end{align*}
for $0<\delta<1$, $u\in L^p(\T^2)$, and $v\in X$. 

Next, we reformulate \eqref{lim.eq}. Set $u=\phi-\Delta \phi$. Then
\eqref{lim.eq} can be written as a system of two second-order equations:
\begin{align}
  & -\Delta\phi + \phi = u \quad\mbox{in }\T^2,\ t>0, \label{ex.phi} \\
  & \pa_t u - \mu\Delta u + u = (\na^\perp\phi\cdot\na)(\phi-u) + \mu(u-\phi) + u.
  \label{ex.u}
\end{align}
We employ a fixed-point argument. Let $T>0$ and $R>0$. We introduce the 
spaces $Y=C^0([0,T];X)$ and $B_R=\{u\in Y:\|u-u^0\|_Y\le R\}$, where
$u^0=-\Delta\phi^0+\phi^0\in C^\infty(\T^2)$.
Given $u\in Y\subset C^0([0,T];L^p(\T^2))$, let $\phi\in L^\infty(0,T;W^{2,p}(\T^2))$
be the unique solution to \eqref{ex.phi} satisfying the elliptic estimate
$\|\phi\|_{W^{2,p}(\T^2)}\le C\|u\|_{L^p(\T^2)}$. Then define 
\begin{align*}
  J(u) &= e^{-tA_p}u^0 + \int_0^t e^{(t-s)A_p}F(\phi(s),u(s))ds, 
  \quad\mbox{where} \\
  F(\phi,u) &= (\na^\perp\phi\cdot\na)(\phi-u)+\mu(u-\phi)+u.
\end{align*}
Using the continuous embedding $W^{2,p}(\T^2)\hookrightarrow W^{1,2p}(\T^2)$
and the elliptic estimate for $\phi$, we infer the estimate 
\begin{align*}
  \|F(\phi,u)\|_{L^\infty(0,T;L^p(\T^2))}
  &\le C\|u\|_{L^\infty(0,T;W^{1,2p}(\T^2))}
  \big(1+\|u\|_{L^\infty(0,T;W^{1,2p}(\T^2))}\big) \\
  &\le C\|u\|_{L^\infty(0,T;X)}\big(1+\|u\|_{L^\infty(0,T;X)}\big) \\
  &= C\|u\|_Y(1+\|u\|_Y).
\end{align*}
The last inequality follows from the embedding $X\hookrightarrow W^{1,2p}(\T^2)$
which holds for $\beta>1/2+1/(2p)$.

We show that $J$ maps $B_R$ into $B_R$ and that $J:B_R\to B_R$ is a contraction
for sufficiently small $T>0$. Let $T>0$ be such that 
$\|(e^{-tA_p}-I)u_0\|_{L^p(\T^2)}\le CT^\beta\|u^0\|_X\le R/2$. 
Then, for $u\in B_R$,
\begin{align*}
  \|J(u)-u^0\|_Y 
  &\le \sup_{0<t<T}\|(e^{-tA_p}-I)u^0\|_{L^p(\T^2)} \\
  &\phantom{xx}{}+ \sup_{0<t<T}\int_0^t\|A_pe^{-tA_p}F(\phi(s),u(s))\|_X ds \\
  &\le \frac{R}{2} 
  + \sup_{0<t<T}\int_0^t (t-s)^{-\beta}e^{-\delta(t-s)}\|F(\phi(s),u(s))\|_X ds \\
  &\le \frac{R}{2} + \frac{CT^{1-\beta}}{1-\beta}\|u\|_Y(1+\|u\|_Y)
  \le R,
\end{align*}
if $T>0$ is sufficiently small, using that $u\in B_R$. Thus $J(u)\in B_R$.
In a similar way, we show that, for given $u$, $v\in B_R$,
$$
  \|J(u)-J(v)\|_Y \le \frac{CT^{1-\beta}}{1-\beta}(\|u\|_Y+\|v\|_Y)\|u-v\|_Y.
$$
Again, choosing $T>0$ small enough, $J$ becomes a contraction, and the
fixed-point theorem of Banach provides the existence and uniqueness of a
mild solution on $[0,T]$.

It remains to prove that the mild solution is smooth. Since
$\beta>1-1/p$, we have $X\hookrightarrow W^{2,p/2}(\T^2)$ and hence
$u\in L^\infty(0,T;W^{2,p/2}(\T^2))\subset L^\infty(0,T;W^{1,p}(\T^d))$.
Furthermore, $\na\phi\in L^\infty(0,T;W^{1,p}(\T^2))
\subset L^\infty(0,T;L^\infty(\T^2))$ (here, we use $p>2$). 
This shows that $\pa_t u + A_p(u) \in L^\infty(0,T;L^p(\T^2))$. 
Parabolic theory implies that $u\in L^q(0,T;W^{2,p}(\T^2))$ for all $q<\infty$. 
This improves the regularity of $\phi$ to $\phi\in L^q(0,T;W^{4,p}(\T^2))$. 
Hence, $\pa_t u+A_p(u)\in L^q(0,T;L^\infty(\T^2))$, 
and a bootstrap procedure finishes the proof.
\end{proof}
\end{appendix}

%%%%%%%%%%%%%%%%%%%%%%%%%%%%%%%%%%%%%%%%%%%%%%%%%%%%%%%%%%%%%%%%%%%%%%%%%%%%%%%%%%%

\end{document}